\documentclass[
11pt,                          
english                        
]{article}

\usepackage[english]{babel}    
\usepackage{amsmath}           
\usepackage[utf8]{inputenc}    
\usepackage[T1]{fontenc}       
\usepackage{longtable}         
\usepackage{exscale}           
\usepackage[final]{graphicx}   
\usepackage[sort]{cite}        
\usepackage{array}             
\usepackage{wasysym}           
\usepackage[a4paper]{geometry} 
\usepackage[multiuser]{fixme}  
\usepackage{xspace}            
\usepackage{tikz}              
\usepackage{ifdraft}           
\usepackage[expansion=false    
           ]{microtype}        
\usepackage[nottoc]{tocbibind} 
\usepackage[backref=page,      
           final=true,         
           pdfpagelabels       
           ]{hyperref}         

\RequirePackage{amssymb}           
\RequirePackage{bbm}               
\RequirePackage{mathtools}         
\RequirePackage[amsmath,thmmarks,hyperref]{ntheorem} 
\RequirePackage{mathrsfs}          
\RequirePackage{stmaryrd}          
\RequirePackage{paralist}          
\RequirePackage{aliascnt}          

\renewcommand{\mathbb}[1]{\mathbbm{#1}}

\numberwithin{equation}{section}

\allowdisplaybreaks

\renewcommand{\arraystretch}{1.2}

\let\originalleft\left
\let\originalright\right
\renewcommand{\left}{\mathopen{}\mathclose\bgroup\originalleft}
\renewcommand{\right}{\aftergroup\egroup\originalright}

\renewcommand{\cleardoublepage}{\clearpage\ifodd\c@page\else\vspace*{\fill}\thispagestyle{empty}\newpage\fi}

\theoremheaderfont{\normalfont\bfseries}
\theorembodyfont{\itshape}
\newtheorem{lemma}{Lemma}[section]

\newaliascnt{proposition}{lemma}
\newtheorem{proposition}[proposition]{Proposition}
\aliascntresetthe{proposition}

\newaliascnt{thm}{lemma}
\newtheorem{theorem}[thm]{Theorem}
\aliascntresetthe{thm}

\newaliascnt{corollary}{lemma}
\newtheorem{corollary}[corollary]{Corollary}
\aliascntresetthe{corollary}

\newaliascnt{definition}{lemma}
\newtheorem{definition}[definition]{Definition}
\aliascntresetthe{definition}

\newaliascnt{claim}{lemma}

\aliascntresetthe{claim}

\theorembodyfont{\rmfamily}

\newaliascnt{example}{lemma}

\aliascntresetthe{example}

\newaliascnt{remark}{lemma}
\newtheorem{remark}[remark]{Remark}
\aliascntresetthe{remark}

\newaliascnt{question}{lemma}

\aliascntresetthe{question}

\newaliascnt{conjecture}{lemma}

\aliascntresetthe{conjecture}

\theorembodyfont{\itshape}
\theoremnumbering{Roman}

\theoremheaderfont{\scshape}
\theorembodyfont{\normalfont}
\theoremstyle{nonumberplain}
\theoremseparator{:}
\theoremsymbol{\hbox{$\boxempty$}}
\newtheorem{proof}{Proof}
\theoremsymbol{\hbox{$\triangledown$}}

\theoremheaderfont{\normalfont\bfseries}
\theorembodyfont{\itshape}
\theoremstyle{nonumberplain}
\theoremsymbol{}
\theoremseparator{}

\newcommand{\lemmaxargument}{}

\newcommand{\propositionxargument}{}

\newcommand{\theoremxargument}{}
\newtheorem{theoremxaux}{Theorem \theoremxargument}
\newenvironment{theoremx}[1][]{\renewcommand\theoremxargument{#1}\theoremxaux}{\endtheoremxaux}

\newcommand{\corollaryxargument}{}

\newcommand{\definitionxargument}{}

\newcommand{\claimxargument}{}

\theorembodyfont{\rmfamily}

\newcommand{\examplexargument}{}

\newcommand{\remarkxargument}{}

\newcommand{\exercisexargument}{}

\newcommand{\questionxargument}{}

\newcommand{\conjecturexargument}{}

\newcommand{\chairXaddress}{%
  Institut für Mathematik \\
  Lehrstuhl für Mathematik X \\
  Universität Würzburg \\
  Campus Hubland Nord \\
  Emil-Fischer-Straße 31 \\
  97074 Würzburg \\
  Germany}

\newcommand{\at}[1]          {\big|_{#1}}

\DeclareMathOperator{\id}    {\mathsf{id}}
\newcommand{\pr}             {\mathrm{pr}}
\DeclareMathOperator{\image} {\mathrm{im}}

\newcommand{\lie}[1]          {\mathfrak{#1}}

\DeclareMathOperator{\ad}     {\mathrm{ad}}
\DeclareMathOperator{\Ad}     {\mathrm{Ad}}

\newcommand{\acts}            {\mathbin{\triangleright}}

\newcommand{\Fun}[1][k]      {\mathscr{C}^{#1}}
\newcommand{\Cinfty}         {\Fun[\infty]}
\newcommand{\Lie}   {\mathscr{L}}

\newcommand{\Formen}         {\Omega}
\newcommand{\HdR}            {\mathrm{H}_{\scriptscriptstyle\mathrm{dR}}}

\DeclareMathOperator{\prol}{\mathrm{prol}}

\DeclareMathSymbol\dAlembert     {\mathord}{AMSa}{"03}

\newcommand{\tensor}[1][{}]           {\mathbin{\otimes_{\scriptscriptstyle{#1}}}}

\newcommand{\Anti}                    {\Lambda}
\newcommand{\Sym}                     {\mathrm{S}}

\DeclareMathOperator{\ins}            {\mathrm{i}}

\expandafter\def\csname opt@stmaryrd.sty\endcsname
{only,shortleftarrow,shortrightarrow}
\usepackage{extpfeil}          

\usetikzlibrary{matrix}
\usetikzlibrary{arrows}
\usetikzlibrary{patterns}
\usetikzlibrary{decorations.pathreplacing}

{

  \author{ \textbf{Thorsten
      Reichert}\thanks{\texttt{thorsten.reichert@mathematik.uni-wuerzburg.de}}
    \\[0.5cm]
    \chairXaddress }
}

  \newcommand{\qham}[1]{\mathbf{#1}}
  \newcommand{\J}{\qham{J}}

  \newcommand{\fparam}{\nu}
  \newcommand{\formal}{\llbracket \fparam \rrbracket}
  \newcommand{\Inv}{^{\mathrm{G}}}
  \newcommand{\InvS}[1]{^{\mathrm{G},{#1}}}
  \newcommand{\Equi}{_\lie{g}}
  \newcommand{\Red}{_{\mathrm{red}}}
  \newcommand{\Nice}{_{nice}}

  \newcommand{\drham}{\mathrm{d}}
  \newcommand{\dequi}{\drham\Equi}
  \newcommand{\dkoszul}{\delta}
  \newcommand{\qkoszul}{\partial}

  \newcommand{\Closed}{Z}
  
  \newcommand{\EquiCohom}{\mathrm{H}\Equi}
  \newcommand{\Def}{\mathrm{Def}}

  \newcommand{\qad}{\frac{1}{\fparam}\ad}

  \AtBeginDocument{}

\title{Characteristic classes of star products on Marsden-Weinstein
  reduced symplectic manifolds}

\date{Current Version of equi: \gitAuthorIsoDate\\[0.2cm]
 {\small
   Last changes by \gitAuthorName{} on \gitAuthorDate \\
   Git revision of equi: \texttt{\gitAbbrevHash{}} \gitReferences
 }
}

\date{\today}

\begin{document}

%
%

\maketitle

%
%

\begin{abstract}
    In this note we consider a quantum reduction scheme in deformation
    quantization on symplectic manifolds proposed by Bordemann, Herbig
    and Waldmann based on BRST cohomology. We explicitly construct the
    induced map on equivalence classes of star products which will
    turn out to be an analogue to the Kirwan map in the Cartan model
    of equivariant cohomology. As a byproduct we shall see that every
    star product on a (suitable) reduced manifold is equivalent to a
    reduced star product.
\end{abstract}

\medskip
\noindent
\textbf{Mathematics Subject Classification Primary.} 55N91 \\
\textbf{Mathematics Subject Classification Secondary.} 55N91, 53D20 \\

\medskip
\noindent
\textbf{Keywords.} deformation quantization, marsden-weinstein
reduction, quantum momentum map, equivariant cohomology, cartan model

\medskip
\noindent
The final publication is available at link.springer.com, see \\
\href{http://link.springer.com/article/10.1007\%2Fs11005-016-0921-z}{http://link.springer.com/article/10.1007\%2Fs11005-016-0921-z}

%
%


%
%

\section{Introduction}
\label{sec:introduction}

Since ancient times symmetries have played a pivotal role in both
physics and mathematics. And so has the art of symmetry reduction,
i.e. getting rid of excess degrees of freedom. One prominent example
would, of course, be Marsden-Weinstein reduction on symplectic
manifolds \cite{marsden.weinstein:1974a}, which will be the main focus
of this note. Given a classical system, conceived as a symplectic
manifold $(M,\omega)$, with symmetry given by a Hamiltonian
(i.e. there is an $\ad^*$-equivariant momentum map $J$) action of a
Lie group $G$ by symplectomorphisms, Marsden-Weinstein reduction is a
two-step-process. First, take the level set $C \coloneqq
J^{-1}(\left\{0\right\})$ ($0$ should be a value and regular value of
$J$) of the momentum map, whereupon we assume the induced group action
to be free and proper. This allows one, in the second step, to build
the quotient $M\Red = C / G$ which turns out to be again a symplectic
manifold. The whole situation can be summarized in the diagram
\begin{equation*}
    M \xhookleftarrow{\ \ \iota\ \ } C
    \xtwoheadrightarrow{\pi} M\Red.
\end{equation*}

With the advent of Quantum Mechanics, there have been numerous
proposals of how to implement symmetry reduction in any given Quantum
Theory, starting with Dirac \cite{dirac:1950a}. However, just as with
the multitude of quantization schemes developed over time, even in one
such scheme, there is typically no ``universal'' reduction
process. Here we will investigate only one quantum reduction scheme in
the context of deformation quantization \cite{bayen.et.al:1978a}
proposed by Bordemann, Herbig and Waldmann in
\cite{bordemann.herbig.waldmann:2000a} and further developed in
\cite{gutt.waldmann:2010a} by Gutt and Waldmann, which is based on
BRST cohomology. We will provide a brief recap to the extent needed
later on in \autoref{sec:reduction}. One of the central ingredients of
this reduction scheme is the notion of quantum momentum maps, a direct
generalization of the concept of momentum maps on symplectic manifolds
(see \cite{xu:1998a}, we will mostly follow conventions from
\cite{mueller-bahns.neumaier:2004a, mueller-bahns.neumaier:2004b}):
given a connected Lie group $G$ with Lie algebra $\lie{g}$ acting on a
symplectic manifold $M$ by symplectomorphisms and a star product
$\star$ on $M$, a quantum momentum map is a linear map $\J\colon
\lie{g} \longrightarrow \Cinfty(M)\formal$ into the formal series of
smooth functions on $M$ such that for all $\xi, \eta \in \lie{g}$
\begin{equation*}
    \Lie_{X_\xi} = -\qad_\star\left( \J(\xi) \right) \qquad \text{and}
      \qquad \left[ \J(\xi), \J(\eta) \right]_\star = \fparam\J\left(
          \left[\xi, \eta \right] \right)
\end{equation*}
hold (where we denoted by $X_\xi$ the fundamental vector field of
$\xi$). The pair $(\star, \J)$ is then called an equivariant star
product and the equivalence classes of equivariant star products where
recently shown to be characterized by the second equivariant
cohomology (in the Cartan model, see \cite{cartan:1951a,
  guillemin.sternberg:1999a, kalkman:1993a}) $\EquiCohom^2(M)$ of $M$
with respect to $G$ in \cite{reichert.waldmann:2016a}. Star products
on the Marsden-Weinstein reduced symplectic manifold, which we will
throughout denote by $M\Red$, on the other hand, are classified by the
second de Rham cohomology $\HdR^2(M\Red)$, see
\cite{bertelson.cahen.gutt:1997a, deligne:1995a, fedosov:1996a,
  gutt.rawnsley:1999a, nest.tsygan:1995a, weinstein.xu:1998a} for the
symplectic and \cite{kontsevich:2003a} for the more general Poisson
case.

The main question we will be answering is the following: given any
equivariant star product $(\star, \J)$ on $M$ with characteristic
class $c\Equi(\star, \J) \in \frac{1}{\fparam} \EquiCohom^2(M)\formal$
and the corresponding reduced star product $\star\Red$ with
characteristic class $c(\star\Red) \in \frac{1}{\fparam}
\HdR^2(M\Red)\formal$, what exactly is the relation between these
classes?  A previous result by Bordemann \cite{bordemann:2005a}
already gives a partial answer. Using the representation of
equivariant differential forms as equivariant maps $\lie{g}
\longrightarrow \Formen(M)$ (see \autoref{sec:cohomology}) gives maps
\begin{equation*}
    \EquiCohom^2(M) \xrightarrow{\mathrm{ev}_0}
    \HdR\InvS{2}(M) \stackrel{i}{\longrightarrow}
    \HdR^2(M)
\end{equation*}
where $\HdR\InvS{2}(M)$ denotes the second invariant de Rham
cohomology of $M$ with respect to the action of $G$ (note that, for
noncompact $G$, this is different from the invariant part of the de
Rham cohomology), $\mathrm{ev}_0$ is induced by the evaluation at $0
\in \lie{g}$ and $i$ is induced by the inclusion of invariant
differential forms into differential forms. Both are compatible with
taking (equivariant, invariant) characteristic classes of star
products, that is the following diagram commutes
\cite{reichert.waldmann:2016a}
\begin{center}
    \begin{tikzpicture}
        \matrix (m) [matrix of math nodes, ampersand replacement=\&,
        column sep=2.5em, row sep=2.5em] {
          \mathrm{Star}\Equi(M) \&
          \&
          \mathrm{Star}(M) \\
          \frac{1}{\fparam}\EquiCohom^2(M)\formal \&
          \frac{1}{\fparam}\HdR\InvS{2}(M)\formal \&
          \frac{1}{\fparam}\HdR^2(M)\formal \\
        };
        \path[->] (m-1-1) edge (m-1-3);
        \path[->] (m-2-1) edge node[below] {$\mathrm{ev}_0$} (m-2-2);
        \path[->] (m-2-2) edge node[below] {$i$} (m-2-3);
        \path[->] (m-1-1) edge node[left] {$c\Equi$} (m-2-1);
        \path[->] (m-1-3) edge node[right] {$c$} (m-2-3);
    \end{tikzpicture}
\end{center}
where the top map is the inclusion of equivariant star products into
star products on $M$.  One can then compare the characteristic classes
of $\star$ and $\star\Red$ on the momentum level set $C$ used in the
classical Marsden-Weinstein reduction via pullbacks
\begin{equation*}
    \HdR(M) \stackrel{\iota^*}{\longrightarrow} \HdR(C)
    \stackrel{\pi^*}{\longleftarrow} \HdR(M\Red)
\end{equation*}
and one finds that $\iota^* c(\star) = \pi^* c(\star\Red)$
\cite{bordemann:2005a}. However, even for nontrivial $\HdR^2(M)$ or
$\HdR^2(M\Red)$, there are cases where $\HdR^2(C) = 0$ and thus this
equation does not provide any insights. One such example is given by
the Hopf-fibration
\begin{equation*}
    \mathbb{C}^{n+1} \setminus \left\{0\right\} \xhookleftarrow{\quad}
    S^{2n+1}
    \xtwoheadrightarrow{} \mathbb{C}\mathbb{P}^n.
\end{equation*}
To alleviate this problem, we will throughout \autoref{sec:cohomology}
construct a map $K\colon \EquiCohom(M) \longrightarrow \HdR(M\Red)$
which circumvents the projection $\EquiCohom(M) \longrightarrow
\HdR(M)$ and enables us to prove the main theorem in
\autoref{sec:classification}
\begin{theoremx}[(Main theorem)]
    Let $M$ be a symplectic manifold equipped with a smooth and proper
    Hamiltonian $G$-action for a finite dimensional, connected Lie
    group $G$ and let $J\colon M \longrightarrow \lie{g}^*$ be the
    corresponding $\Ad^*$-equivariant momentum map. Assume furthermore
    that the induced action of $G$ on $J^{-1}(\left\{0\right\})$ is
    free. Given any equivariant star product $(\star,\J)$ on $M$ and
    the corresponding reduced star product $\star\Red$ on $M\Red$, we
    then have
    \begin{equation*}
        K\left( c\Equi(\star, \J) \right) = c(\star\Red).
    \end{equation*}
Furthermore, $K$ is surjective.
\end{theoremx}
The map $K$ will turn out to be the Cartan model analogue of the
Kirwan map \cite{kirwan:1984}, which is defined for the topological,
or Borel, model otherwise known as the homotopy quotient. The critical
remark here is that we will not restrict ourselves to compact Lie
groups and hence the cohomologies of the Cartan and Borel model
typically do not agree. The reason we are using the Cartan model at
all is of course the fact, that it classifies (also in the noncompact
case) equivariant star products on symplectic manifolds
\cite{reichert.waldmann:2016a}.

During the construction of $K$ one pivotal result will be that for any
$G$-principal bundle $P \stackrel{\pi}{\longrightarrow} B$ the
equivariant cohomology (in the Cartan model) of the total space is
related to the de Rham cohomology of the base by $\pi^* \colon \HdR(B)
\cong \EquiCohom(P)$ \cite{cartan:1951a,
  guillemin.sternberg:1999a}. Since in the context of
Marsden-Weinstein reduction the action of $G$ on the momentum level
set $C$ is proper and free, $C$ can be viewed as a principal bundle
over $M\Red$ \cite{duistermaat.kolk:2000a}, which enables us to write
$K$ concisely as
\begin{equation*}
    K = \left( \pi^* \right)^{-1} \circ \iota^*\colon \EquiCohom(M)
    \stackrel{\iota^*} \longrightarrow \EquiCohom(C) \xrightarrow{\left(
          \pi^* \right)^{-1}} \HdR(M\Red)
\end{equation*}
Finally, the surjectivity of $K$ shows that any star product on
$M\Red$ is equivalent to one obtained by quantum reduction from $M$.

\noindent
\textbf{Acknowledgements:} The author would like to thank Stefan
Waldmann for numerous helpful discussions, James Stasheff for useful
advice on the preprint, the anonymous referee for their constructive
input, Jonas Schnitzer for assistance with proofreading the text as
well as Marco Benini and Alexander Schenkel for valuable remarks.

\section{Reduction of Star Products}
\label{sec:reduction}

Throughout this exposition, let $(M,\omega)$ be a connected symplectic
manifold equipped with a smooth and proper Hamiltonian $G$-action for
a finite dimensional, connected Lie group $G$ and let $J\colon M
\longrightarrow \lie{g}^*$ be the corresponding $\Ad^*$-equivariant
momentum map. We will furthermore require that $0$ is a value and
regular value of $J$ and that the $G$-action on
$J^{-1}(\left\{0\right\})$ is free and proper. In this setting we can
apply Marsden-Weinstein reduction to obtain the reduced symplectic
manifold $M\Red \coloneqq J^{-1}(\left\{0\right\}) / G$. We then have
the following maps
\begin{equation*}
    M \xhookleftarrow{\ \ \iota\ \ } C \coloneqq J^{-1}(\left\{0\right\})
    \xtwoheadrightarrow{\pi} M\Red
\end{equation*}
where $\iota$ is an inclusion of a closed submanifold and $\pi$ a
surjective submersion. The symplectic two-form $\omega\Red$ on $M\Red$
is uniquely determined by $\iota^*\omega = \pi^*\omega\Red$. We will
frequently summarize the above situation by stating that $M\Red$ is
Marsden-Weinstein reduced \cite{marsden.weinstein:1974a} from $M$ via
$C$ (for details see e.g\ \cite{waldmann:2007a}).

For the convenience of the reader we will briefly recall a
construction from \cite{gutt.waldmann:2010a} to obtain star products
on $M\Red$ from star products on $M$ (see also \cite{bordemann:2005a,
  bordemann.herbig.waldmann:2000a, cattaneo.felder:2004a,
  cattaneo.felder:2007a, forger.kellendonk:1992a, lu:1993a}).  First,
since the action of $G$ is proper, there exists an open neighbourhood
$M\Nice \subseteq M$ of $C$ together with a $G$-equivariant
diffeomorphism
\begin{equation*}
    \Phi\colon M\Nice \longrightarrow U\Nice \subseteq C \times
    \lie{g}^* \qquad
    \text{with} \qquad \mathrm{pr}_1 \circ \Phi \circ \iota = \id_C
\end{equation*}
onto an open neighbourhood $U\Nice$ of $C \times \left\{0\right\}$,
where the $G$-action on $C\times \lie{g}^*$ is the product action of
the one on $C$ and $\Ad^*$, such that for each $p\in C$ the subset
$U\Nice \cap \left( \left\{p\right\} \times \lie{g}^* \right)$ is star
shaped around $\left\{p\right\} \times \left\{0\right\}$ and the
momentum map $J$ is given by the projection onto the second factor,
i.e. $J\at{M\Nice} = \pr_2 \circ \Phi$
\cite{bordemann.herbig.waldmann:2000a}. We can use $\Phi$ to define
the following prolongation map:
\begin{equation}
\label{eq:prol}
    \prol\colon \Cinfty(C) \longrightarrow \Cinfty(M\Nice)\colon \phi
    \longmapsto \left( \pr_1 \circ \Phi \right)^*\ \phi
\end{equation}
Clearly we have $\iota^* \circ \prol = \id_{\Cinfty(C)}$. Next
consider the (classical) Koszul complex, given by
\begin{equation*}
    \Cinfty\left( M, \Anti^\bullet_\mathbb{C} \lie{g}\right) =
    \Cinfty(M) \tensor \Anti^\bullet_\mathbb{C} \lie{g} \qquad \text{with} \qquad
    \dkoszul = \ins(J).
\end{equation*}
$U\Nice$ being star shaped allows to define
\begin{equation*}
    (h_k x)(p) = e_a \wedge \int\limits_0^1 t^k \frac{\partial\left( x \circ
      \Phi^{-1} \right)}{\partial \mu_a}(c, t\mu) \mathrm{d} t
\end{equation*}
for $x\in \Cinfty(M,\Anti^k \lie{g})$ where we chose a basis
$\left\{e_a\right\}$ of $\lie{g}$ and denoted $\Phi(p) = (c,\mu)$. The
following proposition \cite[Prop. 2.1]{gutt.waldmann:2010a} summarizes
some properties of $h_k$:
\begin{proposition}
    The Koszul complex $\left( \Cinfty(M\Nice, \Anti^\bullet\lie{g}),
        \dkoszul \right)$ is acyclic with explicit homotopy $h$ and
    homology $\Cinfty(C)$ in degree $0$. In detail, we have
    \begin{equation*}
        h_{k-1}\dkoszul_k + \dkoszul_{k+1} h_k = \id_{\Cinfty(M\Nice,
          \Anti \lie{g})}
    \end{equation*}
    for $k \geq 0$ and
    \begin{equation*}
        \prol \iota^* + \dkoszul_1 h_0 = \id_{\Cinfty(M\Nice)}
    \end{equation*}
    as well as $\iota^* \dkoszul_1 = 0$. Thus the Koszul complex is a
    free resolution of $\Cinfty(C)$ as $\Cinfty(M\Nice)$-modules. We have
    \begin{equation*}
        h_0 \prol = 0
    \end{equation*}
    and all the homotopies $h_k$ are $G$-equivariant.
\end{proposition}

Turning towards quantum reduction, we will exclusively be interested
in equivariant (formal) star products on $M$, so let us give a quick
definition (compare \cite{hamachi:2002a},
\cite{mueller-bahns.neumaier:2004a}, \cite{xu:1998a}):
\begin{definition}
    A (formal) star product on $(M,\omega)$ is a bilinear map
    \begin{equation*}
        \star\colon \Cinfty(M) \times \Cinfty(M) \longrightarrow
        \Cinfty(M)\formal\colon (f,g) \longmapsto f\star g =
        \sum\limits_{k=0}^\infty \fparam^k C_k(f,g),
    \end{equation*}
    such that its $\fparam$-linear extension to $\Cinfty(M)\formal$ is
    an associative product, all $C_k$ are bidifferential operators,
    $C_0(f,g) = fg$ and $C_1(f,g) - C_1(g,f) = \left\{ f,g
    \right\}_\omega$ for all $f,g \in \Cinfty(M)$.  An equivariant
    star product is a pair $(\star, \J)$ of a star product $\star$
    together with a linear map $\J\colon \lie{g} \longrightarrow
    \Cinfty(M)\formal$ such that
    \begin{equation*}
        \Lie_\xi = -\qad_\star(\J(\xi)) \qquad \text{and} \qquad
        \J([\xi,\eta]) = \left[ \J(\xi), \J(\eta)
        \right]_\star
    \end{equation*}
    where we denoted by $\Lie_\xi$ the Lie derivative with respect to
    the fundamental vector field $X_\xi$ of $\xi$.
\end{definition}
The definition of an equivariant star product immediately implies that
$\Lie_\xi$ is a derivation of $\star$ and, since $G$ is assumed to be
connected, that $G$ acts by $\star$-automorphisms. Here we are using
the convention $\ad_\star(f)(g) \coloneqq \left[ f,g \right]_\star$
with $\left[\ ,\,\right]_\star$ being the commutator with respect to
$\star$.

We will start by introducing the quantized Koszul operator
\cite{gutt.waldmann:2010a}:
\begin{definition}[Quantized Koszul operator]
    Let $\kappa \in \mathbb{C}\formal$. The quantized Koszul operator\\
    $\qkoszul^{(\kappa)}\colon \Cinfty(M, \Anti^\bullet_\mathbb{C}
    \lie{g})\formal \longrightarrow \Cinfty(M, \Anti^{\bullet -
      1}_\mathbb{C} \lie{g})\formal$ is defined by
    \begin{equation*}
        \qkoszul^{(\kappa)}x = \ins(e_a)x \star \J_a +
        \frac{\fparam}{2} C_{ab}^c e_c \wedge \ins(e_a)\ins(e_b)x +
        \fparam \kappa \ins(\Delta)x
    \end{equation*}
    where $C_{ab}^c = e^c\left( \left[ e_a,e_b \right] \right)$ are
    the structure constants of $\lie{g}$ and
    \begin{equation}
        \Delta(\xi) = \mathrm{tr}\ad(\xi) \qquad \text{for } \xi \in \lie{g}
    \end{equation}
    is the modular one-form $\Delta \in \lie{g}^*$ of $\lie{g}$.
\end{definition}
Here $\left\{e_a\right\}$ is assumed to be any basis of $\lie{g}$,
$\J_a \coloneqq \J(e_a)$ and $\ins(\xi) x$ denotes the insertion of
any $\xi \in \lie{g}$ into the first argument of $x \in
\Anti^\bullet_\mathbb{C}\lie{g}$. We will from now on fix $\kappa$ and
omit any explicit mention in all subsequent formulae. Some properties
of $\qkoszul$ are collected in \cite[Lemma~3.4]{gutt.waldmann:2010a}:
\begin{lemma}
    Let $(\star, \J)$ be an equivariant star product and $\kappa \in
    \mathbb{C}\formal$. Then one has
    \begin{enumerate}[i)]
    \item $\qkoszul$ is left $\star$-linear.
    \item The classical limit of $\qkoszul$ is
        $\dkoszul$.
    \item $\qkoszul$ is $G$-equivariant.
    \item $\qkoszul \circ \qkoszul = 0$.
    \end{enumerate}
\end{lemma}
Following \cite{bordemann.herbig.waldmann:2000a} one can introduce a
deformation of the classical restriction map $\iota^*$ by
\begin{equation*}
    I^* = \iota^* \left( \id + \left(
            \qkoszul_1 - \dkoszul_1 \right) h_0
    \right)^{-1}\colon \Cinfty(M)\formal \longrightarrow
    \Cinfty(J^{-1}(\left\{0\right\}))\formal
\end{equation*}
where $h$ is a homotopy of the classical Koszul complex. Furthermore,
one can find a homotopy $H$ with
$H_{-1} = \prol$ such that the
augmented complex with $\qkoszul_0 =
I^*$ has trivial homology:
\begin{equation*}
    H_{k-1} \qkoszul_k +
    \qkoszul_{k+1} H_k =
    \id_{\Cinfty(M, \Anti^\bullet \lie{g})\formal}
\end{equation*}
for $k\geq 0$ and $I^* \prol =
\id_{\Cinfty(C)\formal}$ for $k=-1$. Moreover the maps
$I^*$ and $H_k$ are
$G$-equivariant.
To finally arrive at the reduced star product, one defines a left
$\star$-ideal $\mathcal{J}_C$ and its normalizer $\mathcal{B}_C$
\begin{equation*}
    \begin{split}
        \mathcal{J}_C &\coloneqq \image(\qkoszul_1)
        \subseteq \Cinfty(M)\formal \\
        \mathcal{B}_C &\coloneqq \left\{ f \in \Cinfty(M)\formal\
        \right|\left.\ \left[ f, \mathcal{J}_C \right]_\star \subseteq
            \mathcal{J}_C \right\}
    \end{split}
\end{equation*}
to obtain the mutually inverse maps
\begin{equation}
\label{eq:reduction:functions}
    \begin{matrix}
        \mathcal{B}_C / \mathcal{J}_C &
        \longrightarrow & \pi^*
        \Cinfty(M\Red)\formal &
        : &
        [f] &
        \longmapsto &
        I^* f
        \\
        \Cinfty(M\Red)\formal &
        \longrightarrow &
        \mathcal{B}_C / \mathcal{J}_C &
        : &
        u &
        \longmapsto &
        [ \prol(\pi^* u)]
        \\
    \end{matrix}
\end{equation}
which enable us to define
\begin{equation*}
    \pi^*\left( u \star\Red v \right) \coloneqq
    I^* \left( \prol (\pi^* u) \star \prol (\pi^*
        v) \right)
\end{equation*}
for all $u,v \in \Cinfty(M\Red)\formal$.

Since we are interested mainly in classifying equivariant star
products and their corresponding reduced star products, the first
critical property to check is whether equivariantly equivalent star
products on $M$ reduce to equivalent star products on $M\Red$.
\begin{lemma}
    Let $T\colon (\star^1, \qham{J}^1) \longmapsto (\star^2,
    \qham{J}^2)$ be an equivariant equivalence, then
    \begin{equation*}
        T\Red \coloneqq \left( (\pi^*)^{-1} \circ
            I^* \right)
        \circ T \circ \left( \prol \circ \pi^* \right)
    \end{equation*}
    is an equivalence $T\Red\colon \star^{1}\Red \longmapsto
    \star^{2}\Red$.
\end{lemma}
\begin{proof}
    First of all, let us check that $T$ induces a map $\mathcal{B}^1_C
    / \mathcal{J}^1_C \longrightarrow \mathcal{B}_C^2 /
    \mathcal{J}_C^2$. By extending $T$ onto $\Cinfty(M, \Anti^\bullet
    \lie{g}) \cong \Cinfty(M) \tensor \Anti^\bullet \lie{g}$ as the
    identity on the second factor, we can calculate for any $x \in
    \Cinfty(M, \Anti^\bullet \lie{g})\formal$:
    \begin{equation*}
        \begin{split}
            T\qkoszul^1 x &= T\left( \ins(e^a)x \star^1 J_a^1 + \frac{\fparam}{2}
                C_{ab}^c e_c \wedge \ins(e^a)\ins(e^b) x + \fparam
                \kappa \ins(\Delta) x \right) \\
            &= \ins(e^a) Tx \star^2 TJ_a^1 + \frac{\fparam}{2}
                C_{ab}^c e_c \wedge \ins(e^a)\ins(e^b) Tx + \fparam
                \kappa \ins(\Delta) Tx \\
            &= \qkoszul^2(Tx) \\
        \end{split}
    \end{equation*}
    since $T\qham{J}^1 = \qham{J}^2$ and $T$ commutes with all
    insertions and wedge products of Lie algebra elements. This shows
    in particular, that $T$ is a chain map between the two quantized
    Koszul complexes
    \begin{equation*}
        T\colon \left( \Cinfty(M, \Anti^\bullet \lie{g})\formal,\ \qkoszul^1
        \right) \longrightarrow \left( \Cinfty(M, \Anti^\bullet
            \lie{g})\formal,\ \qkoszul^2 \right).
    \end{equation*}
    Thus for any $f = \qkoszul^1 x$ we know that $Tf = \qkoszul^2 T x$
    and hence $Tf \in \mathcal{J}^2_C$. Even more, since $T$ is
    invertible, $\mathcal{J}_C^1 \cong \mathcal{J}_C^2$ (as sets)
    holds. Take then any $j_2 \in \mathcal{J}_C^2$, any $f \in
    \mathcal{B}_C^1$, define $j_1 \coloneqq T^{-1} j_2$, and calculate
    \begin{equation*}
        [Tf, j_2]_{\star^2} = [Tf, Tj_1]_{\star^2} =
        T[f,j_1]_{\star^1} \in T\mathcal{J}_C^1 = \mathcal{J}_C^2,
    \end{equation*}
    hence we have $\mathcal{B}_C^1 \cong \mathcal{B}_C^2$ (as sets).
    Furthermore, by \eqref{eq:reduction:functions} and the fact that
    $T$ is an equivalence and thus starts with
    $\id_{\Cinfty(M)\formal}$ in 0th order, we know that $T\Red$ also
    has $\id_{\Cinfty(M\Red)\formal}$ in 0th order. The only thing
    left to check is then that the higher orders of $T\Red$ are
    differential operators on $M\Red$. This however is a direct
    consequence from the fact that $I^*$ can be decomposed into
    \cite{bordemann.herbig.waldmann:2000a}
    \begin{equation*}
        I^* = \iota^* \circ \left( \id + \sum\limits_{k=1}^\infty
            \fparam^k S_k \right)
    \end{equation*}
    with differential operators $S_k$.
\end{proof}

\section{Equivariant Cohomology on Principal Fibre Bundles}
\label{sec:cohomology}

As seen in \cite{reichert.waldmann:2016a}, equivariant star products
on symplectic manifolds are classified by the second equivariant
cohomology (or, to be more precise, by the cohomology of the Cartan
complex of equivariant differential forms). In the context of
Marsden-Weinstein reduction (\autoref{sec:reduction}) we will be
interested mostly in the equivariant cohomology of principal bundles,
more specifically, the principal bundle $\pi\colon C =
J^{-1}(\left\{0\right\}) \longrightarrow M\Red$ (which is a principal
bundle since the action on $C$ is free and proper, see
\cite{duistermaat.kolk:2000a}).

To start off, we will first recall the necessary basic definitions of
equivariant cohomology (for a detailed exposition consult e.g.\
\cite{guillemin.sternberg:1999a, cartan:1951a}). Let $M$ be a manifold
equipped with a smooth $G$-action for any finite dimensional,
connected Lie group $G$ with Lie algebra $\lie{g}$ and consider the
complex of equivariant differential forms
\begin{equation*}
    \left( \Formen\Equi^k(M) \coloneqq \bigoplus\limits_{2i+j=k} \left[
            \Sym^i(\lie{g}^*) \tensor \Formen^j(M) \right]\Inv, \quad
        \dequi = \drham + \ins_\bullet \right)
\end{equation*}
where $\Sym$ denotes the symmetric tensor algebra, $\Formen$ the de
Rham complex, $\drham$ the de Rham differential, $\ins_\bullet$ the
insertion of fundamental vector fields of the action into the
differential form part and invariants are taken with respect to the
tensor product of the coadjoint action $\Ad^*$ of $G$ on
$\Sym(\lie{g}^*)$ and the pullback on $\Formen(M)$
\begin{equation*}
    \acts \colon G \times \left( \Sym(\lie{g}^*) \tensor \Formen(M) \right)
    \longrightarrow \Sym(\lie{g}^*) \tensor \Formen(M) \colon
    \left(g, p \tensor \alpha\right) \longmapsto \Ad^*(g) p \tensor
    (g^{-1})^* \alpha.
\end{equation*}
We can view elements of $\alpha \in
\left[\Sym(\lie{g}^*) \tensor \Formen(M)\right]\Inv$ as polynomial
maps $\alpha\colon \lie{g} \longrightarrow \Formen(M)$ such that
\begin{center}
    \begin{tikzpicture}
         \matrix (m) [matrix of math nodes, ampersand replacement=\&,
        column sep=2.5em, row sep=2.5em] {
          \lie{g} \& \Formen(M) \\
          \lie{g} \& \Formen(M) \\
        };
        \path[->] (m-1-1) edge node[left] {$\Ad(g)$} (m-2-1);
        \path[->] (m-1-2) edge node[right] {$(g^{-1})^*$} (m-2-2);
        \path[->] (m-1-1) edge node[above] {$\alpha$} (m-1-2);
        \path[->] (m-2-1) edge node[below] {$\alpha$} (m-2-2);
    \end{tikzpicture}
\end{center}
commutes. Occasionally, for any $\alpha \in \left[ \Sym^i(\lie{g}^*)
    \tensor \Formen^j(M) \right]\Inv$, we will refer to $k = 2i + j$
as the total, $i$ the symmetric and $j$ the exterior degree of
$\alpha$. Finally, we will frequently make use of the pullback by
smooth functions on the complex of equivariant differential forms and
the equivariant cohomology, so let us give a brief recap. For any two
manifolds $M$ and $N$ with $G$ actions and any equivariant smooth map
$f\colon M \longrightarrow N$ we define for $p \tensor \alpha \in
\left[ \Sym^i(\lie{g}^*) \tensor \Formen^j(N) \right]\Inv$ the
pullback
\begin{equation*}
    f^*(p \tensor \alpha) = p \tensor f^*\alpha.
\end{equation*}
Clearly, $f^*(p\tensor\alpha) \in \left[ \Sym^i(\lie{g}^*) \tensor
    \Formen^j(M) \right]\Inv$ since $f$ is equivariant.
Also note that $\Formen\Equi$ is a contravariant functor, since it can
be expressed as
\begin{equation*}
    \Formen\Equi = \left[ \left(\_\Inv\right) \circ
        \left(\Sym(\lie{g}^*) \tensor \_\right) \circ \Omega
    \right].
\end{equation*}
Lastly, we will denote by $\EquiCohom(M)$ the cohomology of
$\Formen\Equi(M)$ and note that the map $[p\tensor\alpha]\Equi
\longmapsto [f^*(p\tensor\alpha)]\Equi$ on cohomology is well defined,
since $\ins_{\drham f\bullet} f^* = f^* \ins_\bullet$ and $\drham f^*
= f^* \drham$.
\begin{remark}
    $\Formen\Equi(M)$ in general only computes the equivariant
    cohomology of $M$ under special circumstances (e.g. if $G$ is
    compact or if the action of $G$ on $M$ is free and proper, see
    \autoref{corollary:equi_derham_iso}) and hence is, in general, not
    a model of equivariant cohomology. Thus it is important to note
    that subsequently we will always refer to $\EquiCohom(M)$ by
    equivariant cohomology.
\end{remark}
We will be needing one central result from equivariant cohomology, due
to to Cartan \cite{cartan:1951a}
\begin{theorem}
    Let $C$ be a $G$-principal bundle. Then
    \begin{equation*}
        \EquiCohom(C) \cong \mathrm{H}\left( \Formen_\mathrm{bas}(C),
            \drham \right)
    \end{equation*}
\end{theorem}
For a (more general) proof consult
e.g. \cite{guillemin.sternberg:1999a, kalkman:1993a,
  nicolaescu:1999a}. Here however, we will for the convenience of the
reader, showcase a shorter, more elementary proof. To this end we will
need a well known result about basic differential forms on fibre
bundles. Recall that, for a surjective submersion $\pi \colon M
\longrightarrow N$, a differential form $\mu$ is called basic if
$\ins_Y \mu = 0$ and $\Lie_Y \mu = 0$ for all $Y \in \ker\left(
    T\pi\right)$. We will denote the complex of basic differential
forms on $M$ by $\Formen_\mathrm{bas}(M)$. The following lemma is
widely known:
\begin{lemma}
    \label{lemma:basic_forms}
    Let $\pi\colon M \longrightarrow N$ be a surjective submersion
    such that $\pi^{-1}(y)$ is a connected submanifold of $M$ for all
    $y \in N$. Then a differential form $\mu \in \Formen(M)$ is basic
    if and only if there exists a $\nu \in \Formen(N)$ such that
    \begin{equation*}
        \mu = \pi^* \nu
    \end{equation*}
\end{lemma}

Returning to the equivariant cohomology of $C$, there is one
additional result needed which involves principal connections on
principal bundles. The suitable definition of principal connections on
a $G$-principal bundle for our purposes is that of a $\lie{g}$-valued
1-form $\omega \in \Formen^1(P) \tensor \lie{g}$ with
\begin{equation*}
    \Ad_g\left( (g^{-1})^* \omega \right) = \omega \qquad \text{and}
    \qquad \omega\left(X_\xi\right) = \xi
\end{equation*}
for all $g \in G$ and $\xi \in \lie{g}$ (again, $X_\xi$ denotes the
fundamental vector field of $\xi$). Let us from now on fix an
arbitrary principal connection $\omega \in \Formen^1(C) \tensor
\lie{g}$ (existence is guaranteed e.g. by \cite{atiyah:1957a} or
\cite{kobayashi.nomizu:1963a}). We can then evaluate any $p \in
\Sym^1(\lie{g})$ on the second tensor factor of $\omega$, for which we
will write $p(\omega) \in \Formen^1(C)$. We will now use $\omega$ to
define for $k \geq 1$ the following map
\begin{equation}
    \label{eq:cartan_vert_homotopy}
    h_\omega \colon \Sym^k(\lie{g}^*) \tensor \Formen(C) \longrightarrow
    \Sym^{k-1}(\lie{g}^*) \tensor \Formen(C)\colon
    \prod\limits_{i=1}^{k} p_i \tensor \alpha \longmapsto
    \sum\limits_{j=1}^k \ \prod\limits_{\substack{i=1 \\ i \neq
            j}}^k p_i \tensor p_j(\omega) \wedge \alpha
\end{equation}
and $h_\omega = 0$ on $\Sym^0(\lie{g}^*) \tensor \Formen(C)$.
\begin{lemma}
\label{lemma:cartan_vert_homotopy}
    $h_\omega$ is a contraction of the chain complex $C^k = \left[
        \Sym^k(\lie{g}^*) \tensor \Formen^{n-k}(C) \right]\Inv$ with
    differential $\ins_\bullet$ (using the convention $\Formen^n(C) =
    0$ for $n < 0$):
    \begin{equation*}
        \ins_\bullet h_\omega + h_\omega \ins_\bullet = \id
    \end{equation*}
\end{lemma}
\begin{proof}
    The proof has two parts. First, we have to show that $h$ is a
    $G$-equivariant map and secondly, that $\ins_\bullet h_\omega +
    h_\omega\ins_\bullet = \id$ holds. To avoid notational clutter, we
    will perform calculations only for $k=1$. All other cases are
    straightforward generalizations thereof.  So let $g \in G$ and $p
    \tensor \alpha \in \Sym^1(\lie{g}^*) \tensor \Formen(C)$. Then we
    know by the equivariance property of $\omega$ that
    \begin{equation*}
        \begin{split}
            h_\omega\left( g \acts p \tensor \alpha \right) &=
            h_\omega\left( \Ad^*_g p \tensor (g^{-1})^* \alpha \right)
            = \left( \Ad^*_g p \right)(\omega) \wedge (g^{-1})^*
            \alpha \\
            &= p\left( \Ad_{g^{-1}} \omega\right) \wedge (g^{-1})^*
            \alpha
            = (g^{-1})^* \left( p(\omega) \wedge \alpha \right) \\
            &= g \acts h_\omega(p \tensor \alpha) \\
        \end{split}
    \end{equation*}
    Finally, we can compute (using $\omega(X_\xi) = \xi$ and therefore
    $\ins_\xi p(\omega) = p(\xi)$ for all $\xi \in \lie{g}$)
    \begin{equation*}
        \ins_\bullet h_\omega(p \tensor \alpha) = \ins_\bullet\left( p(\omega)
            \wedge \alpha \right) = \ins_\bullet p(\omega) \wedge \alpha -
        p(\omega) \wedge \ins_\bullet \alpha = p \tensor \alpha -
        h_\omega\left( \ins_\bullet p\tensor\alpha \right)
    \end{equation*}
\end{proof}

The significance of the previous lemma becomes clear, once we view the
complex of equivariant differential forms as a double complex
$\Formen^{i,j}\Equi(C) = \left[ \Sym^i(\lie{g}^*) \tensor \Formen^j(C)
\right]\Inv$, with vertical differential $\ins_\bullet$ and horizontal
differential $\drham$:
\begin{center}
    \begin{tikzpicture}
        \matrix (m) [matrix of math nodes, row sep=2.5em, column
        sep=2em, ampersand replacement=\&] {
          \Formen^{2,n-3}\Equi(C)
          \&
          \Formen^{2,n-2}\Equi(C)
          \&
          \Formen^{2,n-1}\Equi(C)
          \&
          \&
          \\
          \&
          \Formen^{1,n-2}\Equi(C)
          \&
          \Formen^{1,n-1}\Equi(C)
          \&
          \Formen^{1,n}\Equi(C)
          \&
          \\
          \&
          \&
          \Formen^{n-1}(C)\Inv
          \&
          \Formen^{n}(C)\Inv
          \&
          \Formen^{n+1}(C)\Inv
          \\
        };
        \path[transform canvas={xshift=-0.5ex},->] (m-2-2) edge
        node[below left] {$\ins_\bullet$} (m-1-1);
        \path[transform canvas={xshift=-0.5ex},->] (m-3-3) edge node[below left] {$\ins_\bullet$}
        (m-2-2);
        \path[transform canvas={xshift=-0.5ex},->] (m-2-3) edge (m-1-2);
        \path[transform canvas={xshift=-0.5ex},->] (m-3-4) edge (m-2-3);
        \path[transform canvas={xshift=-0.5ex},->] (m-2-4) edge (m-1-3);
        \path[transform canvas={xshift=-0.5ex},->] (m-3-5) edge (m-2-4);
        \path[->] (m-3-3) edge node[below] {$\drham$} (m-3-4);
        \path[->] (m-3-4) edge node[below] {$\drham$} (m-3-5);
        \path[->] (m-2-2) edge (m-2-3);
        \path[->] (m-2-3) edge (m-2-4);
        \path[->] (m-1-1) edge (m-1-2);
        \path[->] (m-1-2) edge (m-1-3);
        \path[transform canvas={xshift=0.5ex},->] (m-1-1) edge
        node[above right] {$h_\omega$} (m-2-2);
        \path[transform canvas={xshift=0.5ex},->] (m-2-2) edge
        node[above right] {$h_\omega$} (m-3-3);
        \path[transform canvas={xshift=0.5ex},->] (m-1-2) edge
        (m-2-3);
        \path[transform canvas={xshift=0.5ex},->] (m-2-3) edge
        (m-3-4);
        \path[transform canvas={xshift=0.5ex},->] (m-1-3) edge
        (m-2-4);
        \path[transform canvas={xshift=0.5ex},->] (m-2-4) edge
        (m-3-5);
    \end{tikzpicture}
\end{center}
\autoref{lemma:cartan_vert_homotopy} then shows that the columns of
$\Formen^{\bullet,\bullet}(C)$ are exact and hence, by a general
argument about double complexes with exact columns (see e.g.\
\cite{bott.tu:1982a}), we know that the total cohomology, which is
precisely the equivariant cohomology, is given by the horizontal
cohomology of the kernel of the vertical differential in the bottom
row.
\begin{corollary}
    \label{corollary:equi_derham_iso}
    Let $G$ be a connected Lie group and $C$ a $G$-principal
    bundle. Then
    \begin{equation*}
        \EquiCohom(C) \cong \HdR(C / G)
    \end{equation*}
\end{corollary}
\begin{proof}
    Let $\alpha \in \Omega^k\Equi(C)$ be $\dequi$-closed and let
    $\alpha_l$ be the component of $\alpha$ with maximal symmetric
    degree $l$. Then, since $\alpha$ is closed we must have
    $\ins_\bullet \alpha_l = 0$ and therefore, by
    \autoref{lemma:cartan_vert_homotopy}, there must be a $\beta_l$
    with $\ins_\bullet \beta_l = \alpha_l$. By subtracting $\dequi
    \beta_l$ from $\alpha$ its cohomology class stays the same,
    however $\alpha - \dequi \beta_l$ has maximal symmetric degree
    $l-1$ or less. Repeating this process, one can find in every
    cohomology class a representative of symmetric degree
    zero.\\
    Now, the bottom row complex of $\Formen\Equi(C)$ is just the
    complex of invariant differential forms $\left( \Formen(C)\Inv,\
        \drham \right)$. Consequently all $\dequi$-closed forms in the
    bottom row complex are those, that are invariant, $\drham$-closed
    and $\ins_\bullet$-closed, which is equivalent to being basic and
    $\drham$-closed. Since the bundle projection $\pi\colon C
    \longrightarrow C / G$ is a surjective submersion, $\pi^*\colon
    \Formen_{\mathrm{bas}}(C) \longrightarrow \Formen(C/G)$ is a chain
    isomorphism, hence
    \begin{equation*}
        \EquiCohom(C) \cong \HdR(C / G).
    \end{equation*}
\end{proof}

\begin{remark}
\label{remark:kirwan_on_closed}
    Since the rows of $\Formen\Equi(P)$ are not only exact, but exact
    by a given homotopy $h_\omega$, we can consider the following map
    (denote by $\Closed\Equi$ ($\Closed_\mathrm{bas}$) closed
    equivariant (basic) differential forms)
    \begin{equation*}
        \phi \colon \Closed\Equi(C) \longrightarrow
        \Closed\Equi(C) \colon \alpha \longmapsto \alpha -
        \dequi h_\omega \alpha.
    \end{equation*}
    Obviously, $\phi$ induces $\id_{\EquiCohom(C)}$ on
    cohomology. However, on representatives, $\phi$ reduces the
    maximal symmetric degree of $\alpha$ by at least one and therefore
    implements the algorithm used for
    \autoref{corollary:equi_derham_iso} to reduce $\alpha$ to a basic
    form on $P$ (additionally $\phi$ alters the lower degrees,
    too. This however is not important here). We can now use $\phi$ to
    define
    \begin{equation*}
        \Phi \coloneqq \prod\limits_{k=1}^\infty \phi \colon
        \Closed\Equi(P) \longrightarrow \Closed\Equi(P).
    \end{equation*}
    Since $\Phi$ stabilizes on $\Closed\Equi^{k, \bullet}(C)$ after at
    most $k$ applications of $\phi$, there are no convergence problems
    present. Of course, $\Phi$ also induces the identity on
    cohomology. The important part however is that $\image \Phi
    \subseteq \Closed_\mathrm{bas}(C)$.
\end{remark}

\begin{remark}
    \label{remark:proj_quasi_iso}
    From \autoref{corollary:equi_derham_iso} it is clear that $\pi^*
    \colon \Formen(C/G) \longrightarrow \Formen\Equi(C)$ is a
    quasi-isomorphism of differential graded associative algebras.
\end{remark}

\begin{corollary}
\label{corollary:kirwan_map}
Let $M\Red$ be Marsden-Weinstein reduced from $M$ via $C$ by the
action of a finite-dimensional, connected Lie group $G$. Then the map
\begin{equation*}
    K\colon \Closed\Equi(M) \longrightarrow \Closed(M\Red) \colon
    K = (\pi^*)^{-1} \circ \Phi \circ \iota^*
\end{equation*}
is well-defined and induces
\begin{equation*}
    K\colon \EquiCohom^2(M) \longrightarrow \HdR^2(M\Red) \colon
    K = (\pi^*)^{-1} \circ \iota^*
\end{equation*}
on cohomology.
\end{corollary}
\begin{proof}
    Apply \autoref{corollary:equi_derham_iso},
    \autoref{remark:kirwan_on_closed} and
    \autoref{remark:proj_quasi_iso} to the case $C/G \cong M\Red$.
\end{proof}

\begin{remark}
    The map $K$ from \autoref{corollary:kirwan_map} can bee seen as
    the Cartan-model analogue of the Kirwan map
    \cite{kirwan:1984}. Again, we emphasize that we cannot use the
    original Kirwan map since we are working with not necessarily
    compact Lie groups.
\end{remark}

The intriguing question here is of course what we can say about the
image of $K$, which is completely determined by the image of
$\iota^*$.
\begin{corollary}
    Let $M\Red$ be Marsden-Weinstein reduced from $M$. Then $K\colon
    \EquiCohom^2(M) \longrightarrow \HdR^2(M\Red)$ from
    \autoref{corollary:kirwan_map} is surjective.
\end{corollary}
\begin{proof}
    The very definition of $\prol$ \eqref{eq:prol} extends to the de Rham-complexes of
    $M$ and $C$:
    \begin{equation*}
        \prol \colon \Formen(C) \longrightarrow \Formen(M) \colon
        \prol = (\mathrm{pr}_1 \circ \Phi)^*
    \end{equation*}
    and we clearly have $\iota^* \circ \prol =
    \id_{\Formen(C)}$. Furthermore, by functoriality of
    $\Sym(\lie{g}^*) \tensor \bullet$, $\vphantom{\Formen}\Inv$ and
    cohomology, this equation holds on equivariant cohomology. Thus
    $\iota^*$ has a right-inverse and hence must be surjective.
\end{proof}

\section{Characteristic Classes of reduced Star Products}
\label{sec:classification}

Having the results of the previous sections at hand, we can proceed to
prove the main theorem of this paper. It relies heavily on
\autoref{corollary:equi_derham_iso}, the classification of
(equivariant) star products on symplectic manifolds and a result from
\cite{bordemann:2005a} which relates the characteristic class of a
star product with the characteristic class of its reduction. Let us
begin by recalling the relevant classification results. On one hand,
the set of equivalence classes of star products on symplectic
manifolds up to equivalences of star products $\Def(M,\omega)$ is
isomorphic to formal power series in the second de Rham cohomology of
the manifold (see \cite{bertelson.cahen.gutt:1997a},
\cite{deligne:1995a}, \cite{gutt.rawnsley:1999a})
\begin{equation*}
    c\colon \Def(M,\omega) \stackrel{\sim}{\longrightarrow}
    \frac{\omega}{\fparam} + \HdR^2(M)\formal
\end{equation*}
while on the other hand, the set of equivalence classes of equivariant
star products deforming a momentum map $J$ up to equivariant
equivalences $\Def(M,\omega,J)$ is isomorphic to power series in the
second equivariant cohomology, see \cite{reichert.waldmann:2016a}
\begin{equation*}
    c\Equi\colon \Def(M,\omega, J) \stackrel{\sim}{\longrightarrow}
    \frac{\omega - J}{\fparam} + \EquiCohom^2(M)\formal.
\end{equation*}
Both $c$ and $c\Equi$ are bijections.  One can even give explicit
expressions of both characteristic classes for the case of
(equivariant) Fedosov star products (see \cite{fedosov:1994a}), which
are essentially all (equivariant) star products (by
\cite{bertelson.cahen.gutt:1997a},
\cite{reichert.waldmann:2016a}). Strictly speaking, the Fedosov
construction maps pairs of a torsion-free, symplectic connection
$\nabla$ and a formal series of closed two-forms $\Omega \in \fparam
\Closed^2(M)\formal$ to star products. We will however fix once and
for all torsion-free, symplectic (and invariant, if applicable)
connections on all manifolds involved and henceforth drop all
references to them. Instead, we will denote Fedosov star products
constructed from $\Omega$ by $F(\Omega)$ for which then the following
equations hold:
\begin{equation*}
    c(F(\Omega)) = \frac{1}{\fparam} [\omega + \Omega]
    \qquad \qquad
    c\Equi(F(\Omega), \J) = \frac{1}{\fparam} [ \omega + \Omega -
    \J]\Equi.
\end{equation*}
Finally, from \cite{bordemann:2005a} we have the lemma
\begin{lemma}
    \label{lemma:bordemann}
    Let $M\Red$ be Marsden-Weinstein reduced from $M$ via $C$ with
    inclusion $\iota\colon C \longrightarrow M$ and principal bundle
    projection $\pi\colon C \longrightarrow M\Red$. Additionally, let
    $(\star, \J)$ be an equivariant star product on $M$ and let
    $\star\Red$ be the corresponding reduced star product on
    $M\Red$. Then we have
    \begin{equation*}
        \iota^* c(\star) = \pi^* c(\star\Red).
    \end{equation*}
\end{lemma}

Using all those results, we obtain
\begin{theorem}
\label{theorem:main}
    The characteristic class $c(\star\Red)$ of $\star\Red$ is given by
    \begin{equation*}
        c(\star\Red) = K\left(c\Equi(\star,\J')\right).
    \end{equation*}
\end{theorem}
\begin{proof}
    For the $(-1)$-th order in $\fparam$ this follows directly from
    the Marsden-Weinstein reduction since $\left( \J\at{\fparam = 0}
    \right)\at{C} = 0$ and $\iota^* \omega = \pi^* \omega\Red$. Thus
    $\iota^*(\omega - \J\at{\fparam = 0}) = \iota^* \omega$ is basic
    and $\omega\Red$ is the unique form on $M\Red$ with
    $\pi^*\omega\Red = \iota^*\omega$. For the higher orders, let
    $F(\Omega\Red)$ be a Fedosov star product equivalent to
    $\star\Red$ (for its existence see
    \cite{bertelson.cahen.gutt:1997a}) and let $(F(\Omega), \J)$ be a
    Fedosov star product equivariantly equivalent to $(\star, \J')$
    (which exists due to \cite{reichert.waldmann:2016a}). Now
    observe that
    \begin{equation*}
        \begin{split}
            \iota^*c(F(\Omega),\J)_+ & \coloneqq \iota^* \left(
                c\Equi(F(\Omega),\J) - \left[ \frac{\omega -
                      \J\at{\fparam=0}}{\fparam} \right]\Equi \right)
            = \frac{1}{\fparam}[\iota^*(\Omega - \J_+)]\Equi \\
            \pi^* c(F(\Omega\Red))_+ & \coloneqq \pi^* \left(
                c(F(\Omega\Red)) - \left[ \frac{\omega\Red}{\fparam}
                \right] \right)
            = \frac{1}{\fparam} [\pi^*\Omega\Red]\Equi \\
        \end{split}
    \end{equation*}
    where we denoted by $\J_+$ the terms of order strictly greater $0$
    in $\fparam$. Using $\widetilde{\Omega} = K\left(
        c\Equi(F(\Omega),\J)_+ \right)$ as a shorthand notation, we
    know from the definition of $K$ (see
    \autoref{corollary:kirwan_map}) that
    $\iota^*c\Equi(F(\Omega),\J)_+ =
    [\pi^*\widetilde{\Omega}]\Equi$. Hence we have
    \begin{equation*}
        [\iota^*(\Omega - \J)]\Equi = [\pi^* \widetilde{\Omega}]\Equi
    \end{equation*}
    which is equivalent (by using the definition of $\dequi$) to
    \begin{equation*}
        \iota^* \Omega - \pi^* \widetilde{\Omega} = \drham \theta
        \qquad \text{and} \qquad \ins_\bullet \theta = -\J_+
    \end{equation*}
    for some $\theta \in \Formen^1(C)\formal$. Additionally, we know
    from \autoref{lemma:bordemann}, \cite{bordemann:2005a} that
    $\iota^*\Omega - \pi^*\Omega\Red = \drham \mu$ for some $\mu \in
    \Formen^1(C)\formal$. Combining those two yields
    \begin{equation*}
        \pi^*\widetilde{\Omega}
        - \pi^*\Omega\Red = \drham(\mu - \theta).
    \end{equation*}
    Here the left hand side tells us that the form is basic, while
    from the right hand side, we see that it is exact. Hence we can
    infer the existence of $\chi \in \Formen^1(M\Red)\formal$ such
    that $\pi^*(\widetilde{\Omega} - \Omega\Red) = \drham
    \pi^*\chi$. But this immediately shows that
    $[\pi^*\widetilde{\Omega}]\Equi = [\pi^*\Omega\Red]\Equi$ and in
    turn
    \begin{equation*}
        \iota^*c\Equi(F(\Omega), \J)_+
        = \frac{1}{\fparam} [\iota^*(\Omega - \J)]\Equi
        = \frac{1}{\fparam} [\pi^*\widetilde{\Omega}]\Equi
        = \frac{1}{\fparam} [\pi^*\Omega\Red]\Equi
        = \pi^*c(F(\Omega\Red))_+
    \end{equation*}
    Finally, remember that $K$ is precisely $(\pi^*)^{-1} \circ
    \iota^*$ to conclude the proof.
\end{proof}

With \autoref{corollary:kirwan_map} and \autoref{theorem:main} we can
deduce the following corollaries:

\begin{corollary}
    Let $M\Red$ be Marsden-Weinstein reduced from $M$ by $G$. Then for
    any star product $\star$ on $M\Red$, there exists a
    $G$-equivariant star product $(\widetilde{\star}, \J)$ on $M$ such
    that $\star$ and $(\widetilde{\star},\J)\Red$ are equivalent.
\end{corollary}
\begin{proof}
    $K$ is surjective and two star products on $M\Red$ are equivalent
    if and only if their characteristic classes coincide
    \cite{bertelson.cahen.gutt:1997a}.
\end{proof}

\begin{corollary}
\label{corollary:universal_reduction}
If $M\Red$ can be Marsden-Weinstein reduced from $M$ by $G$ and the
second invariant de Rham cohomology $\HdR\InvS{2}(M)$ vanishes, then
for any star product $\star$ on $M\Red$ there exists a quantum
momentum map $\J$ of $F(0)$ such that $(F(0), \J)\Red$ is equivalent
to $\star$.
\end{corollary}
\begin{proof}
    First, $F(0)$ is invariant, see
    \cite{mueller-bahns.neumaier:2004a}.  Since $\HdR\InvS{2}(M) = 0$
    any two invariant star products are invariantly equivalent and
    hence every invariant star product is invariantly equivalent to
    $F(0)$. But then every equivariant star product $(\star, \J')$ is
    equivariantly equivalent to $(\star_M, T\J')$ whenever $T$ is an
    invariant equivalence between $\star$ and $\star_M$.
\end{proof}

Especially the second corollary should be reminiscent of
\cite{gotay.tuynman:1989a} wherein it is shown that $\mathbb{R}^n$ is
(up to a cohomological condition) universal with respect to reduction,
that is almost every symplectic manifold $M\Red$ can be obtained as a
reduction of $\mathbb{R}^n$. Here, every star product on a symplectic
manifold $M\Red$ that has been Marsden-Weinstein reduced from $M$,
can be obtained as a reduction of $F(0)$ as long as $\HdR\InvS{2}(M)$
vanishes . The main difference (and drawback) with
\autoref{corollary:universal_reduction} is of course that it is
restricted to Marsden-Weinstein reduction only whereas in\cite{gotay.tuynman:1989a} reduction with respect to coisotropic
submanifolds is used and it is, to the authors knowledge, not clear
which symplectic manifolds arise as Marsden-Weinstein reductions from
$\mathbb{R}^n$. On the other hand, reduction of star products by
coisotropic manifolds seems to be difficult and only partial results
are known (compare \cite{bordemann:2005a, cattaneo.felder:2004a,
  cattaneo.felder:2007a, gloessner:1998a:pre}).  Also, one can easily
see that the condition $\HdR\InvS{2}(M)$ in
\autoref{corollary:universal_reduction} poses a real obstruction as
seen in the example of $\mathbb{R}^n$ acting on $\mathbb{R}^n$ by
translations.

%
%

{
 \footnotesize
 \renewcommand{\arraystretch}{0.5}

}

%
%

\ifdraft{\clearpage}
\ifdraft{\phantomsection}
\ifdraft{\addcontentsline{toc}{section}{List of Corrections}}
\ifdraft{\listoffixmes}

%
%

\end{document}
